\documentclass[12pt,a4paper]{article}      

\usepackage{graphicx}
\usepackage{enumerate}
\usepackage{amssymb, amsmath, amsthm}

\newtheorem{satz}{Theorem}[section]

\newtheorem{cor}[satz]{Corollary}

\newtheorem{bem}[satz]{Observation}

\usepackage[T1]{fontenc}
\usepackage{lmodern}

\newcommand{\R}{\ensuremath{{\mathbb R}}}

\newcommand{\Pv}{\mathbb{P}}
\newcommand{\Ex}{\mathbb{E}}
\newcommand{\abs}[1]{\left\lvert#1 \right\rvert}
\newcommand{\norm}[1]{\left \lVert#1 \right\rVert}

\begin{document}

\title{On the Maximum of Random Variables on Product Spaces}
\author{Joscha Prochno \and Stiene Riemer}
\date{\today}
\maketitle

\begin{abstract} 
Let $\xi_i$, $i=1,...,n$, and $\eta_j$, $j=1,...,m$ be iid p-stable respectively q-stable random variables, $1< p < q<2$. We prove estimates for $\Ex_{\Omega_1} \Ex_{\Omega_2}\max_{i,j}\abs{a_{ij}\xi_i(\omega_1)\eta_j(\omega_2)}$ in terms of the $\ell_p^m(\ell_q^n)$-norm of $(a_{ij})_{i,j}$. Additionally, for p-stable and standard gaussian random variables we prove estimates in terms of the $\ell_p^m(\ell_{M_{\xi}}^n)$-norm, $M_{\xi}$ depending on the Gaussians. Furthermore, we show that a sequence $\xi_i$, $i=1,\ldots,n$ of iid $\log-\gamma(1,p)$ distributed random variables ($p\geq 2$) generates a truncated $\ell_p$-norm, especially $\Ex \max_{i}\abs{a_i\xi_i}\sim \norm{(a_i)_i}_2$ for $p=2$. As far as we know, the generating distribution for $\ell_p$-norms with $p\geq 2$ has not been known up to now.
\end{abstract}


\section{Introduction and Notation}\label{intro} Let $\xi_i$, $i=1,...,n$ be independent copies of a random variable $\xi$ on a probability space $(\Omega_1,\mathfrak{A}_1,\Pv_1)$, whose first moment is finite. Furthermore, let $a_i$, $i=1,...,n$ be real numbers. In \cite{carsten2} and \cite{carsten1}, the following theorem was shown:

\begin{satz}\label{thm:orlicz_carsten1} Let
	\begin{equation}\label{eqn:m1}
	  M_{\xi}(s) = \int_0^s \frac{1}{t} \Pv_1\left(|\xi| \geq \frac{1}{t}\right) + \int_{\frac{1}{t}}^{\infty} \Pv_1(|\xi| \geq u) du ~ dt.
	\end{equation}
	Then, for all $x\in\R^n$,
$$\Ex\max\limits_{i=1,...,n}|a_i\xi_i|\sim \norm{(a_i)_{i=1}^n}_{M_{\xi}}.$$
\end{satz}

We recall that a convex function $M:[0,\infty)\rightarrow[0,\infty)$ with $M(0)=0$ is called an Orlicz function. For an Orlicz function $M$ we define the Orlicz norm $\norm{\cdot}_M$ on $\R^n$ by
	$$\norm{x}_M=\inf\left\{t>0\left|\sum\limits_{i=1}\limits^nM\left(\frac{|x_i|}{t}\right)\leq 1\right.\right\},$$
and the Orlicz space $\ell_M^n$ to be $\R^n$ equipped with the norm $\norm{\cdot}_M$. For references see for example \cite{orlicz}.

In the following let also $n_j$, $j=1,...,m$ be independent copies of a random variable $\eta$ on a probability space $(\Omega_2,\mathfrak{A}_2,\Pv_2)$, whose first moment is finite and $a_{ij}$, $i=1,...,n$, $j=1,...,m$ be real numbers. It is a natural question if we can give estimates for 
\begin{equation}\label{eqn:ausdruck}
\Ex_{\Omega_1}\Ex_{\Omega_2}\max\limits_{i,j}\left|a_{ij}\eta_j(\omega_2)\xi_i({\omega_1})\right|.
\end{equation}
Since the random variables $\left(\xi_i\eta_j\right)_{i,j=1}^{n,m}$ are no longer independent on the product space $(\Omega_1\times\Omega_2, \mathfrak{A}_1\otimes\mathfrak{A}_2,\Pv_1\otimes \Pv_2)$ the previous result, Theorem \ref{thm:orlicz_carsten1}, is not applicable in this case.\\
We give precise estimates up to absolute constants for a certain class of random variables, namely $p$- and $q$-stable, $p,q\in (1,2), p<q$, and standard gaussians. This shows in addition that we can treat dependent random variables with a certain structure of dependence and give precise estimates, this has not been feasible at all by now. Considering p-stable random variables seems to be natural in this case, since they generate the $\ell_p$-norm, that means the Orlicz function resulting in Theorem \ref{thm:orlicz_carsten1} equals $s\mapsto s^p$ for $p\in (1,2)$. One would expect, that the standard gaussians generate the $\ell_2$-norm, but in fact, as shown for example in \cite{carsten1}, they do not, but we can treat them as well. These estimates can be found in the second section. For applications we refer the reader to \cite{carsten3}, \cite{carsten2} and \cite{carsten1}.\\ 
Furthermore, in this context the question arose which random variables generate the $\ell_2$-norm, since standard gaussians astonishingly do not. We provide the solution together with the solution of the generation of truncated $\ell_p$-norms ($p>2$) in the third section. Additionally, we give order estimates for (\ref{eqn:ausdruck}) for these generating distributions.\\
${}$\\
In the following we will give order estimates and this will be denoted by $\sim$, since we are not interested in the exact values of the absolute constants. If for example the absolute constants depend on a certain variable $p$ we denote this by $\sim_p$.

\section{Estimates for $\Ex_{\Omega_1}\Ex_{\Omega_2}\max\limits_{i,j}\left|a_{ij}\eta_j(\omega_2)\xi_i({\omega_1})\right|$}

In the following let $p,q\in(1,2)$ with $p<q$.
We analyze $\Ex_{\Omega_1}\Ex_{\Omega_2}\max\limits_{i,j}\left|a_{ij}\eta_j(\omega_2)\xi_i({\omega_1})\right|$ under two different assumptions. Let $\eta$ always be a $p$-stable random variable. In the first case let $\xi$ be a $q$-stable random variable, we prove the following:
$$\Ex_{\Omega_1}\Ex_{\Omega_2}\max\limits_{i,j}\left|a_{ij}\eta_j(\omega_2)\xi_i({\omega_1})\right|\sim\norm{\left(\norm{\left(a_{ij}\right)_{i=1}^n}_q\right)_{j=1}^m}_p.$$
In the second case let $\xi$ be a standard gaussian random variable, we prove under this assumption
$$\Ex_{\Omega_1}\Ex_{\Omega_2}\max\limits_{i,j}\left|a_{ij}\eta_j(\omega_2)\xi_i({\omega_1})\right|\sim\norm{\left(\norm{\left(a_{ij}\right)_{i=1}^n}_{M_{\xi}}\right)_{j=1}^m}_p,$$
where $\norm{\cdot}_{M_{\xi}}$ denotes the Orlicz norm given by the Orlicz function
\begin{equation}\label{eqn:m2}
  M_{\xi}(s) = 
   \left\{ 
    \begin{array}{ll}
                 0 & ,\hbox{if}\quad s=0  \\
                 e^{-\frac{3}{2s^2}} & ,\hbox{if}\quad s\in (0,1)\\
                 e^{-\frac{3}{2}}(3s-2)  & ,\hbox{if}\quad s\geq 1.
    \end{array} 
   \right.
 \end{equation}
The idea to prove these two results is using the triangle inequality and Jensen's inequality for getting a lower and an upper bound. Afterwards we show that the resulting expressions are equal up to constants depending only on $p$ and $q$ using Theorem \ref{thm:orlicz_carsten1}. Furthermore, we show that we can express this resulting object in terms of a product norm, as above. This also allows us, in these cases, to express a result due to S. Kwapien and C. Sch\"utt, \cite{carsten4} (Example 1.6), in terms of random variables and in a very handy form.\\
${}$\\
Applying the results from \cite{carsten3}, combined with the first steps of the proof of Theorem \ref{thm:p/q-stabil}, one obtains
\begin{eqnarray*}
	\left.\begin{split}
	c_1 \alpha^{-1}\norm{\left(\max\limits_{1\leq l\leq n}\left(\frac{n+1-j}{\sum\limits_{i=1}\limits^n\frac{1}{a_{ij}}}\right)\right)_{j=1}^n}_p&\leq E_{\Omega_1}E_{\Omega_2}\max\limits_{i}\max\limits_{j}\left|a_{ij}\eta_j(\omega_2)\xi_i({\omega_1})\right|\\
	&\leq  c_p \beta^{-1}\ln(n+1)\norm{\left(\max\limits_{1\leq l\leq n}\left(\frac{n+1-j}{\sum\limits_{i=1}\limits^n\frac{1}{a_{ij}}}\right)\right)_{j=1}^n}_p,
	\end{split}\right.
\end{eqnarray*}
where $\eta$ is $p$-stable and $\xi$ is a standard gaussian. Since there is a logarithmic factor in the upper bound, this obviously does not give the correct order. With our method we give the correct order up to absolute constants in a very handy form.

\begin{satz}\label{thm:p/q-stabil} Let $p,q\in(1,2)$ with $p<q$.
	Additionally, let $\xi_i$, $i=1,...,n$ be independent copies of a q-stable random variable $\xi$ on $(\Omega_1,\mathfrak{A}_1,\Pv_1)$ and let $\eta_j$, $j=1,...,m$ be independent p-stable copies of a random variable $\eta$ on $(\Omega_2,\mathfrak{A}_2,\Pv_2)$. Then, for all $(a_{ij})_{i,j}\in\R^{n\times m}$, 
	$$\Ex_{\Omega_1}\Ex_{\Omega_2}\max\limits_{i,j}\left|a_{ij}\eta_j(\omega_2)\xi_i({\omega_1})\right|\sim_{p;q}\norm{\left(\norm{\left(a_{ij}\right)_{i=1}^n}_q\right)_{j=1}^m}_p.$$
\end{satz}

\begin{proof} Let $a_j$, $j=1,...,m$ be real numbers. In \cite{carsten1} it was shown that 
	\begin{equation}\label{eqn:p-stabil}\Ex_{\Omega_2}\max\limits_{1\leq j\leq m}\left|a_j\eta_j(\omega_2)\right|\sim\norm{(a_j)_{j=1}^m}_p.
	\end{equation}
	Applying this, we get
	$$\Ex_{\Omega_1}\Ex_{\Omega_2}\max\limits_{i,j}\left|a_{ij}\eta_j(\omega_2)\xi_i({\omega_1})\right|\sim \Ex_{\Omega_1}\norm{\left(\max\limits_{1\leq i\leq n}\left|a_{ij}\xi_i(\omega_1)\right|\right)_{j=1}^m}_p.$$
	Using the triangle inequality and (\ref{eqn:p-stabil}) for the q-stable $\xi_i$, $i=1,...,n$, we get
	$$\Ex_{\Omega_1}\Ex_{\Omega_2}\max\limits_{i,j}\left|a_{ij}\eta_j(\omega_2)\xi_i({\omega_1})\right|\gtrsim\norm{\left(\Ex_{\Omega_1}\max\limits_{1\leq i\leq n}\left|a_{ij}\xi_i(\omega_1)\right|\right)_{j=1}^m}_p\sim \norm{\left(\norm{\left(a_{ij}\right)_{i=1}^n}_q\right)_{j=1}^m}_p.$$
	For the upper bound we apply Jensen's inequality and obtain
	$$\Ex_{\Omega_1}\Ex_{\Omega_2}\max\limits_{i,j}\left|a_{ij}\eta_j(\omega_2)\xi_i({\omega_1})\right|\lesssim \norm{\left(\left(\Ex_{\Omega_1}\max\limits_{1\leq i\leq n}\left|a_{ij}^p\xi_i^p(\omega_1)\right|\right)^{\frac{1}{p}}\right)_{j=1}^m}_p.$$
	By Theorem \ref{thm:orlicz_carsten1} we get
	$$\Ex_{\Omega_1}\max\limits_{1\leq i\leq n}\left|a_{ij}^p\xi_i^p(\omega_1)\right|\sim\norm{\left(a_{ij}^p\right)_{i=1}^m}_{M_{\xi^p}},$$
	where
	\begin{equation}\label{eqn:m^p}
	M_{\xi^p}(s)=\int_0^s\left(\frac{1}{t}\Pv_1\left(|\xi|^p\geq\frac{1}{t}\right)+\int_{\frac{1}{t}}^{\infty}\Pv_1\left(|\xi|^p\geq u\right)du\right)dt.
\end{equation}
	To prove that the upper and lower bound are equal up to constants, we show that $M_{\xi^p}(s)\sim s^{q/p}$. This is equivalent to $M_{\xi^p}\left(s^p\right)\sim s^{q}$ and hence we get
	$$\norm{\left(a_{ij}^p\right)_{i=1}^n}_{M_{\xi^p}}^{\frac{1}{p}}=\norm{\left(a_{ij}\right)_{i=1}^n}_{M_{\xi^p}\circ ^p}\sim\norm{\left(a_{ij}\right)_{i=1}^n}_q,$$
	since the function $s\mapsto s^q$ generates the $l_q$-norm.
	To do so, we use the fact that if a random variable $\xi$ is $q$-stable for all $t>0$ it holds true that
	\begin{equation}\label{eqn:p-stabil/tail}
		P\left(|\xi|\geq t\right)\lesssim t^{-q},
	\end{equation}
	for references see for example \cite{julio}.\\
	Combining (\ref{eqn:m^p}) and (\ref{eqn:p-stabil/tail}) we get, since $q>p$,
	\begin{eqnarray*}
		\left.\begin{split}
		M_{\xi^p}(s)
		&=\int_0^s\left(\frac{1}{t}\Pv_1\left(|\xi|^p\geq \frac{1}{t}\right)+\int_{\frac{1}{t}}^{\infty}\Pv_1\left(|\xi|^p\geq u\right)du\right)dt\\
		&=\int_0^s\left(\frac{1}{t}\Pv_1\left(|\xi|\geq \left(\frac{1}{t}\right)^{\frac{1}{p}}\right)+\int_{\frac{1}{t}}^{\infty}\Pv_1\left(|\xi|\geq u^{\frac{1}{p}}\right)du\right)dt\\
		&\leq\int_0^s\left(\frac{1}{t}\left(\frac{1}{t}\right)^{-\frac{q}{p}}+\int_{\frac{1}{t}}^{\infty}u^{-\frac{q}{p}}\right)dt\\
		&=\int_0^s\left(t^{-1+\frac{q}{p}}+\left[-u^{-\frac{q}{p}+1}\right]_{\frac{1}{t}}^{\infty}\right)dt\\
		&=2\int_0^st^{-1+\frac{q}{p}}dt\\
		&\sim s^{\frac{q}{p}},
		\end{split}\right.
	\end{eqnarray*}	
which yields the desired result.
	\end{proof}

\begin{satz}\label{thm:p-stabil/gauss} Let $p\in(1,2)$, let $\xi_i$, $i=1,...,n$ be independent copies of a p-stable random variable $\xi$ on $(\Omega_1,\mathfrak{A}_1,\Pv_1)$ and let $\eta_j$, $j=1,...,m$ be independent copies of a standard gaussian random variable $\eta$ on $(\Omega_2,\mathfrak{A}_2,\Pv_2)$. Then, for all $(a_{ij})_{i,j}\in\R^{n\times m}$,
	$$\Ex_{\Omega_1}\Ex_{\Omega_2}\max\limits_{i,j}\left|a_{ij}\eta_j(\omega_2)\xi_i({\omega_1})\right|\sim_{p}\norm{\left(\norm{\left(a_{ij}\right)_{i=1}^n}_{M_{\xi}}\right)_{j=1}^m}_p,$$
	where $\norm{\cdot}_{M_{\xi}}$ denotes the Orlicz norm given by the Orlicz function
	\begin{equation}\label{eqn:m2}
	  M_{\xi}(s) = 
	   \left\{ 
	    \begin{array}{ll}
	                 0 & ,\hbox{if}\quad s=0  \\
	                 e^{-\frac{3}{2s^2}} & ,\hbox{if}\quad s\in (0,1)\\
	                 e^{-\frac{3}{2}}(3s-2)  & ,\hbox{if}\quad s\geq 1.
	    \end{array} 
	   \right.
	 \end{equation}
\end{satz}

Before giving the proof, we need the following observation concerning standard gaussian random variables:

\begin{bem} Let $\xi$ be a standard gaussian random variable, then the following holds for all $t>0$, since the distribution of $\xi$ is symmetric
	$$\Pv\left(|\xi|\geq t\right)=2\Pv\left(\xi\geq t\right)=\sqrt{\frac{2}{\pi}}\int\limits_{t}\limits^{\infty}e^{-\frac{x^2}{2}}dx.$$
Now, applying the results from \cite{szarek1}, we get
\begin{equation}\label{eqn:gauss}\Pv\left(|\xi|\geq t\right)=\sqrt{\frac{2}{\pi}}\int\limits_{t}\limits^{\infty}e^{-\frac{x^2}{2}}\sim \frac{1}{t}e^{-\frac{t^2}{2}}.
\end{equation}
\end{bem}

\begin{proof} (Theorem \ref{thm:p-stabil/gauss}) Let $\left(a_{i}\right)_{i=1}^n\in\R^{n}$. Applying Theorem \ref{thm:orlicz_carsten1}, we get 
	$$\Ex\max\limits_{i=1,...,n}\left|a_{i}\xi_i\right|\sim \norm{(a_i)_{i=1}^n}_{M_{\xi}},$$
	where, as shown in \cite{carsten1}, the following holds
	\begin{eqnarray*}
	  M_{\xi}(s) &= \int\limits_0\limits^s\left( \frac{1}{t} \Pv_1\left(|\xi| \geq \frac{1}{t}\right) + \int\limits_{\frac{1}{t}}\limits^{\infty} \Pv_1(|\xi| \geq u) du\right) dt\\
	   &=\left\{ 
	    \begin{array}{ll}
	                 0 & ,\hbox{if}\quad s=0  \\
	                 e^{-\frac{3}{2s^2}} & ,\hbox{if}\quad s\in (0,1)\\
	                 e^{-\frac{3}{2}}(3s-2)  & ,\hbox{if}\quad s\geq 1.
	    \end{array} 
	   \right.
	 \end{eqnarray*}
	In accordance with the ideas from the proof of Theorem \ref{thm:p/q-stabil}, we get
	$$\Ex_{\Omega_1}\Ex_{\Omega_2}\max\limits_{i,j}\left|a_{ij}\eta_j(\omega_2)\xi_i({\omega_1})\right|\gtrsim \norm{\left(\norm{\left(a_{ij}\right)_{i=1}^n}_{M_{\xi}}\right)_{j=1}^m}_p$$
	and
	$$\Ex_{\Omega_1}\Ex_{\Omega_2}\max\limits_{i,j}\left|a_{ij}\eta_j(\omega_2)\xi_i({\omega_1})\right|\lesssim \norm{\left(\left(\norm{\left(a_{ij}^p\right)_{i=1}^n}_{M_{\xi^p}}\right)^{\frac{1}{p}}\right)_{j=1}^m}_p.$$
	As in the previous proof it remains to show that 
	$$\norm{\left(a_{ij}^p\right)_{i=1}^m}_{M_{\xi^p}}^{\frac{1}{p}}\sim\norm{(a_{ij})_{i=1}^m}_{M_{\xi}}.$$
	Therefore, we prove again that $M_{\xi^p}(s)\sim_p M_{\xi}\left(s^{\frac{1}{p}}\right)$, since this is equivalent to $M_{\xi^p}(s^p)\sim_p M_{\xi}\left(s\right)$ and yields
	$$\norm{\left(a_{ij}^p\right)_{i=1}^n}_{M_{\xi^p}}^{\frac{1}{p}}=\norm{\left(a_{ij}\right)_{i=1}^n}_{M_{\xi^p}\circ ^p}\sim_p\norm{\left(a_{ij}\right)_{i=1}^n}_{M_{\xi}}.$$
	First we show $M_{\xi^p}(s)\lesssim_p M_{\xi}\left(s^{\frac{1}{p}}\right)$ and afterwards we prove the reverse inequality. To do so, we distinguish between $s\leq 1$ and $s> 1$.\\
	${}$\\
	\underline{Upper bound $M_{\xi^p}(s)\lesssim_p M_{\xi}\left(s^{\frac{1}{p}}\right)$:} \\
	${}$\\
	\underline{Case 1:} Let $s\leq 1$.
	\begin{eqnarray*}
		\left.\begin{split}
		M_{\xi^p}(s)
		& =\int\limits_0\limits^s \left(\frac{1}{t} \Pv_1\left(|\xi|^p\geq \frac{1}{t}\right) + \int\limits_{\frac{1}{t}}\limits^{\infty} \Pv_1(|\xi|^p \geq u) du \right) dt\\
		&\sim_p\int_0^{s^{\frac{1}{p}}}\left(\frac{1}{x}\Pv_1\left(|\xi|\geq \frac{1}{x}\right)+\underbrace{\int_{\frac{1}{x}}^{\infty}\Pv_1\left(|\xi|\geq y\right)y^{p-1}dy}_{(I)}\right)x^{p-1}dx.
		\end{split}\right.
	\end{eqnarray*}
	By (\ref{eqn:gauss}) we get
	$$(I)\sim \int_{\frac{1}{x}}^{\infty}e^{-\frac{y^2}{2}}y^{p-2}dy.$$
	Since $p-2<0$ and $y\geq \frac{1}{x}\geq 1$, we have $y^{p-2}\leq 1$ and so we get again by (\ref{eqn:gauss})
	$$(I)\lesssim \int_{\frac{1}{x}}^{\infty}e^{-\frac{y^2}{2}}dy\sim xe^{-\frac{x^2}{2}}.$$
	Altogether 
	\begin{eqnarray*}
		\left.\begin{split}
		M_{\xi^p}(s)
		&\lesssim_p\int_0^{s^{\frac{1}{p}}}\left(\frac{1}{x}\Pv_1\left(|\xi|\geq \frac{1}{x}\right)+xe^{-\frac{1}{2x^2}}\right)x^{p-1}dx.
		\end{split}\right.
	\end{eqnarray*}
	To estimate $P\left(|\xi|\geq \frac{1}{x}\right)$, we apply (\ref{eqn:gauss}) and then take into account that for all $x\in( 0, s^{1/p})$ it holds that $x\leq 1$ and so $e^{-\frac{1}{2x^2}}+xe^{-\frac{1}{2x^2}}\leq 2e^{-\frac{1}{2x^2}}$. Using this, we get
	\begin{eqnarray*}
		\left.\begin{split}
		M_{\xi^p}(s)
		&\lesssim_p\int_0^{s^{\frac{1}{p}}}\left(\frac{1}{x}xe^{-\frac{1}{2x^2}}+xe^{-\frac{1}{2x^2}}\right)x^{p-1}dx\\
		&=\int_0^{s^{\frac{1}{p}}}\left(e^{-\frac{1}{2x^2}}+xe^{-\frac{1}{2x^2}}\right)x^{p-1}dx\\
		&\lesssim\int_0^{s^{\frac{1}{p}}}x^{p-1}e^{-\frac{1}{2x^2}}dx\\
		&=\int_{s^{-\frac{1}{p}}}^{\infty}t^{-p-1}e^{-\frac{t^2}{2}}dt.
		\end{split}\right.
	\end{eqnarray*}
	Since $-p-1<-2$ and $t\geq s^{-\frac{1}{p}}\geq 1$, it holds that $t^{-p-1}\leq 1$. Applying this and (\ref{eqn:gauss}), we get
	\begin{eqnarray*}
		\left.\begin{split}
		M_{\xi^p}(s)
		&\lesssim_p\int_{s^{-\frac{1}{p}}}^{\infty}e^{-\frac{t^2}{2}}dt\sim s^{\frac{1}{p}}e^{-\frac{1}{2s^{\frac{2}{p}}}}.
		\end{split}\right.
	\end{eqnarray*}
	Finally, as $s\leq 1$, we get
	$$M_{\xi^p}(s)\lesssim_p e^{-\frac{1}{2s^{\frac{2}{p}}}}\sim M_{\xi}\left(s^{\frac{1}{p}}\right).$$
		\underline{Case 2:} Let $s> 1$.
		\begin{eqnarray*}
			\left.\begin{split}
		 	M_{\xi^p}(s)
			& =\int_0^s\left(\frac{1}{t}\Pv_1\left(|\xi|^p\geq\frac{1}{t}\right)+\int_{\frac{1}{t}}^{\infty}\Pv_1\left(|\xi|^p\geq u\right)du\right)dt.\\
			& =\underbrace{\int_0^1\left(\frac{1}{t}\Pv_1\left(|\xi|^p\geq\frac{1}{t}\right)+\int_{\frac{1}{t}}^{\infty}\Pv_1\left(|\xi|^p\geq u\right)du\right)dt}_{(a)}\\
			&\quad\quad\quad\quad\quad\quad\quad\quad\quad\quad\quad\quad\quad +\underbrace{\int_1^s\left(\frac{1}{t}\Pv_1\left(|\xi|^p\geq\frac{1}{t}\right)+\int_{\frac{1}{t}}^{\infty}\Pv_1\left(|\xi|^p\geq u\right)du\right)dt}_{(b)}.
			\end{split}\right.
		\end{eqnarray*}
		$(a)$ can be estimated by case 1 and so yields
	$$\int_0^1\left(\frac{1}{t}\Pv_1\left(|\xi|^p\geq\frac{1}{t}\right)+\int_{\frac{1}{t}}^{\infty}\Pv_1\left(|\xi|^p\geq u\right)du\right)dt\lesssim e^{-\frac{1}{2}}.$$
	So it suffices to estimate 
	$$(b)=\underbrace{\int_1^s\frac{1}{t}\Pv_1\left(|\xi|^p\geq\frac{1}{t}\right)dt}_{(I)}+\underbrace{\int_1^s\int_{\frac{1}{t}}^{\infty}\Pv_1\left(|\xi|^p\geq u\right)dudt}_{(II)}.$$
	At first, we estimate $(I)$. Using Markov's inequality, we get
	\textbf{\begin{eqnarray*}
		\left.\begin{split}
		(I)= \int_1^s\frac{1}{t}\Pv_1\left(|\xi|\geq\left(\frac{1}{t}\right)^{\frac{1}{p}}\right)dt
		\leq\int_1^s\frac{1}{t}\frac{\Ex|\xi|}{\left(\frac{1}{t}\right)^{\frac{1}{p}}}dt
		\sim \int_1^st^{-1+\frac{1}{p}}dt=\left[t^{\frac{1}{p}}\right]_1^s
		=s^{\frac{1}{p}}-1
		\stackrel{s\geq 1}{\leq}2s^{\frac{1}{p}}.
		\end{split}\right.
	\end{eqnarray*}}
	To estimate $(II)$, we use (\ref{eqn:gauss}) and get
	\begin{eqnarray*}
		\left.\begin{split}
		(II)&\sim\int_1^s\int_{\frac{1}{t}}^{\infty}u^{-\frac{1}{p}}e^{-\frac{u^{\frac{2}{p}}}{2}}dudt
		=\int_1^s\int_{\frac{1}{t^{\frac{1}{p}}}}^{\infty}y^{-1}e^{-\frac{y^{2}}{2}}py^{p-1}dydt
		\sim_p \int_1^s\int_{\frac{1}{t^{\frac{1}{p}}}}^{\infty}y^{p-2}e^{-\frac{y^{2}}{2}}dydt\\
		&\sim\int_1^s\left[-\Gamma\left(\frac{p-1}{2},\frac{y^2}{2}\right)\right]_{\frac{1}{t^{\frac{1}{p}}}}^{\infty}dt
		=\int_1^s\Gamma\left(\frac{p-1}{2},\frac{1}{2t^{\frac{2}{p}}}\right)dt.
		\end{split}\right.
	\end{eqnarray*}
	In general, we have
	$$\int x^{b-1}\Gamma(t,x)dx=\left[\frac{1}{b}\left(x^b\Gamma(t,x)-\Gamma(t+b,x)\right)\right],$$
	see for example \cite{stegun}. With $b=1$ this provides
	\begin{eqnarray*}
		\left.\begin{split}
		(II)
		&\sim\int_1^s\Gamma\left(\frac{p-1}{2},\frac{1}{2t^{\frac{2}{p}}}\right)dt
		=\left[\frac{1}{2t^{\frac{2}{p}}}\Gamma\left(\frac{p-1}{2},\frac{1}{2t^{\frac{2}{p}}}\right)-\Gamma\left(\frac{p-1}{2}+1,\frac{1}{2t^{\frac{2}{p}}}\right)\right]_1^s\\
		&=\frac{1}{2s^{\frac{2}{p}}}\Gamma\left(\frac{p-1}{2},\frac{1}{2s^{\frac{2}{p}}}\right)-\Gamma\left(\frac{p-1}{2}+1,\frac{1}{2s^{\frac{2}{p}}}\right)+c_p\\
		&\stackrel{s\geq 1}{\leq }\Gamma\left(\frac{p-1}{2},\frac{1}{2s^{\frac{2}{p}}}\right)-\Gamma\left(\frac{p-1}{2}+1,\frac{1}{2s^{\frac{2}{p}}}\right)+c_p.
		\end{split}\right.
	\end{eqnarray*}
	Generally by integration by parts we have
	$$\Gamma(t,x)=(t-1)\Gamma(t-1,x)+x^{t-1}e^{-x}.$$
     We apply this for $t=\frac{p-1}{2}+1$ and $x=\frac{1}{2s^{\frac{2}{p}}}$. Since $1<s<\infty$ and $1<p<2$,
	$$0\leq x^{t-1}e^{-x}=\underbrace{\frac{1}{2s^{\frac{2}{p}}}^{\frac{p-1}{2}+1}}_{\leq\left(\frac{1}{2}\right)^{\frac{1-1}{2}+1}}\underbrace{e^{-\frac{1}{2s^{\frac{2}{p}}}}}_{\leq 1}\leq \frac{1}{2}.$$
	Altogether, this yields
	$$\Gamma\left(\frac{p-1}{2}+1,\frac{1}{2s^{\frac{2}{p}}}\right)\sim_p\Gamma\left(\frac{p-1}{2},\frac{1}{2s^{\frac{2}{p}}}\right)+\tilde{c}_p.$$
	Overall, we have
	$$(II)\lesssim_p \bar{c}_p.$$
	Combining the previous, we get
	\begin{eqnarray*}
		\left.\begin{split}
	 	M_{\xi^p}(s)
		& =(a)+\underbrace{(b)}_{(I)+(II)}
		\lesssim_p e^{-\frac{1}{2}}+s^{\frac{1}{p}}-1+\bar{c}_p
		\stackrel{s\geq 1}{\leq} C_p s^{\frac{1}{p}}
		\sim M_{\xi}\left(s^{\frac{1}{p}}\right).
		\end{split}\right.
	\end{eqnarray*}
	Subsumed, we proved for all $s$
	$$M_{\xi^p}(s)\lesssim_p M_{\xi}\left(s^{\frac{1}{p}}\right).$$
	\underline{Lower bound $M_{\xi^p}(s)\gtrsim_p M_{\xi}\left(s^{\frac{1}{p}}\right)$:} \\
	${}$\\
	\underline{Case 1:} Let $s\leq 1$.	
	\begin{eqnarray*}
		\left.\begin{split}
		M_{\xi^p}(s)
		& =\int\limits_0\limits^s \left(\frac{1}{t} \Pv_1\left(|\xi|^p\geq \frac{1}{t}\right) + \int\limits_{\frac{1}{t}}\limits^{\infty} \Pv_1(|\xi|^p \geq u) du \right) dt\\
		&\sim_p\int_0^{s^{\frac{1}{p}}}\left(\frac{1}{x}\Pv_1\left(|\xi|\geq \frac{1}{x}\right)+\underbrace{\int_{\frac{1}{x}}^{\infty}\Pv_1\left(|\xi|\geq y\right)y^{p-1}dy}_{\geq 0}\right)x^{p-1}dx\\
		&\geq\int_0^{s^{\frac{1}{p}}}x^{p-2}\Pv_1\left(|\xi|\geq \frac{1}{x}\right)dx.
		\end{split}\right.
	\end{eqnarray*}
	We have $-1<p-2<0$ and $x\leq 1$, so $1\leq x^{p-2}\leq x^{-1}$ holds and therefore
	$$M_{\xi^p}(s)\gtrsim_p \int_0^{s^{\frac{1}{p}}}x^{p-2}\Pv_1\left(|\xi|\geq \frac{1}{x}\right)dx\geq\int_0^{s^{\frac{1}{p}}}\Pv_1\left(|\xi|\geq \frac{1}{x}\right)dx.$$
	Applying (\ref{eqn:gauss}), we get
	$$M_{\xi^p}(s)\gtrsim_p\int_0^{s^{\frac{1}{p}}}x e^{-\frac{1}{2x^2}}dx=\int_{s^{-\frac{1}{p}}}^{\infty}t^{-3} e^{-\frac{t^2}{2}}dt=\int_{s^{-\frac{1}{p}}}^{\infty}e^{-\frac{t^2}{2}-3\ln(t)}dt\stackrel{\ln(t)\leq\frac{t^2}{3}}\geq\int_{s^{-\frac{1}{p}}}^{\infty}e^{-\frac{3}{2}t^2}dt\sim s^{\frac{1}{p}}e^{-\frac{3}{2s^{\frac{2}{p}}}}.$$
Finally, we proved
	$$M_{\xi^p}(s)\gtrsim_pM_{\xi}\left(s^{\frac{1}{p}}\right).$$
	\underline{Case 2:} Let $s> 1$.
	\begin{eqnarray*}
		\left.\begin{split}
		M_{\xi^p}(s)
		& =\int\limits_0\limits^s \left(\frac{1}{t} \Pv_1\left(|\xi|^p\geq \frac{1}{t}\right) + \int\limits_{\frac{1}{t}}\limits^{\infty} \Pv_1(|\xi|^p \geq u) du \right) dt\\
		&\sim_p\int_0^{s^{\frac{1}{p}}}\left(\frac{1}{x}\Pv_1\left(|\xi|\geq \frac{1}{x}\right)+\int_{\frac{1}{x}}^{\infty}\Pv_1\left(|\xi|\geq y\right)y^{p-1}dy\right)x^{p-1}dx\\
		&=\underbrace{\int_0^{s^{\frac{1}{p}}}\frac{1}{x}\Pv_1\left(|\xi|\geq \frac{1}{x}\right)x^{p-1}dx}_{\geq 0}+\int_0^{s^{\frac{1}{p}}}\int_{\frac{1}{x}}^{\infty}\Pv_1\left(|\xi|\geq y\right)y^{p-1}dyx^{p-1}dx\\
		&\geq\int_0^{s^{\frac{1}{p}}}\int_{\frac{1}{x}}^{\infty}\Pv_1\left(|\xi|\geq y\right)y^{p-1}dyx^{p-1}dx\\
		&\geq\int_0^{s^{\frac{1}{p}}}\int_{\frac{1}{x}}^{\infty}\Pv_1\left(|\xi|\geq y\right)x^{-p+1}dyx^{p-1}dx\\
		&=\int_0^{s^{\frac{1}{p}}}\int_{\frac{1}{x}}^{\infty}\Pv_1\left(|\xi|\geq y\right)dydx\\
		&\geq\int_1^{s^{\frac{1}{p}}}\int_{\frac{1}{x}}^{\infty}\Pv_1\left(|\xi|\geq y\right)dydx\\
		&\geq\int_1^{s^{\frac{1}{p}}}\int_{1}^{\infty}\Pv_1\left(|\xi|\geq y\right)dydx.
		\end{split}\right.
	\end{eqnarray*}
		By (\ref{eqn:gauss}), we get
	$$\int_{1}^{\infty}\Pv_1\left(|\xi|\geq y\right)dy\sim\int_{1}^{\infty}\frac{1}{y}e^{-\frac{y^2}{2}}dy=\int_{1}^{\infty}e^{-\frac{y^2}{2}-\ln(y)}dy\stackrel{\ln(y)\leq \frac{y^2}{2}}{\geq}\int_{1}^{\infty}e^{-y^2}dy\sim e^{-1},$$
	which yields
	$$M_{\xi}(s)\gtrsim_p \int_1^{s^{\frac{1}{p}}}e^{-1}dx\sim s^{\frac{1}{p}}-1\stackrel{s\geq 1}{\sim}s^{\frac{1}{p}}\sim M_{\xi}(s^{\frac{1}{p}}).$$
	altogether, we proved that for all $s$
	$$M_{\xi^p}(s)\gtrsim_p M_{\xi}\left(s^{\frac{1}{p}}\right).$$
	With regard to the previous, we proved for all $s$
	$$M_{\xi^p}(s)\sim_p M_{\xi}\left(s^{\frac{1}{p}}\right),$$
	which concludes the proof.

\end{proof}

\section{Generation of truncated $\ell_p$-norms ($p>1$)}

Since standard gaussian random variables do not generate the $\ell_2$-norm, the question arises what distribution does. We prove that $\log-\gamma_{1,p}$ ($p>1$) distributed random variables generate more or less the $\ell_p$-norm and especially $\log-\gamma_{1,2}$ distributed random variables generate exactly the $\ell_2$-norm.\\

We remind the reader that the density of a $\log-\gamma_{q,p}$ distributed random variable $\xi$ with parameters $q,p>0$ is given by
  $ f_{\xi}(x) = \left\{%
\begin{array}{ll}
    \frac{p^q}{\Gamma(q)}x^{-p-1}(\ln(x))^{q-1}, & \hbox{$x\geq 1$,} \\
    0, & \hbox{$x<1$.} \\
\end{array}%
\right. $

We prove the following theorem.
   
\begin{satz}
  Let $p>1$ and $\xi_1,\ldots,\xi_n$ be i.i.d. copies of a $\log-\gamma_{1,p}$ distributed random variable $\xi$. Then for all $x\in\R^n$
    $$
      \Ex \max_{1\leq i \leq n} \abs{x_i\xi_i} \sim \norm{x}_{M_{\xi}},
    $$
  and
   $M_{\xi}(s) = \left\{%
\begin{array}{ll}
    \frac{1}{p-1} s^p, & \hbox{$s\leq 1$;} \\
    \frac{p}{p-1} s-1, & \hbox{$s>1$.} \\
\end{array}%
\right.$       
\end{satz}
\begin{proof}
  By Theorem \ref{thm:orlicz_carsten1}, we have
    $$
      \Ex \max_{1\leq i \leq n} \abs{x_i\xi_i} \sim \norm{x}_{M_{\xi}},
    $$
  where
    \begin{equation}\label{EQU Darstellung Orlicz-Funktion}
      M_{\xi}(s) = \int_0^s \frac{1}{t} \Pv\left(|\xi| \geq \frac{1}{t}\right) + \int_{\frac{1}{t}}^{\infty} \Pv(|\xi| \geq u) du ~ dt.
    \end{equation}
  \underline{Case 1:} Let $s\leq 1$. Since we have integration limits $0$ and $s$, $\frac{1}{t}\geq 1$ holds. For all $y\geq 1$
    \begin{equation} \label{EQU Verteilungsfunktion log-gamma}
      \Pv(|\xi| \geq y) = \int_y^{\infty} f_{\xi}(x) dx = \int_y^{\infty} px^{-p-1} dx = [-x^{-p}]_y^{\infty} = y^{-p}. 
    \end{equation}
  Therefore, by (\ref{EQU Verteilungsfunktion log-gamma}) 
    $$
      \int_0^s \frac{1}{t} \Pv\left(|\xi| \geq \frac{1}{t}\right) dt = \int_0^s \frac{1}{t} t^p dt = \frac{s^p}{p}.
    $$
  Furthermore, by (\ref{EQU Verteilungsfunktion log-gamma}) and because $p>1$
    $$
      \int_{\frac{1}{t}}^{\infty} \Pv (|\xi| \geq u) du = \int_{\frac{1}{t}}^{\infty} u^{-p} du = \frac{1}{p-1}t^{p-1}
    $$
  and hence
    $$
      \int_0^s \frac{1}{p-1}t^{p-1} dt = \frac{1}{p(p-1)}s^p.
    $$
  Using the representation (\ref{EQU Darstellung Orlicz-Funktion}), we obtain
    $$
      M_{\xi}(s)=\frac{1}{p}s^p + \frac{1}{p(p-1)}s^p = \frac{1}{p-1}s^p.
    $$
  \underline{Case 2:} Let $s>1$. We first calculate
    $$
      \int_0^s \frac{1}{t} \Pv\left(|\xi| \geq \frac{1}{t}\right) dt.
    $$
  We have
    $$
      \int_0^s \frac{1}{t} \Pv\left(|\xi| \geq \frac{1}{t}\right) dt = 
      \underbrace{\int_0^1 \frac{1}{t} \Pv\left(|\xi| \geq \frac{1}{t}\right) dt}_{(I)} + \underbrace{\int_1^s \frac{1}{t} \Pv\left(|\xi| \geq \frac{1}{t}\right) dt}_{(II)}.
    $$
  In $(I)$, we have $\frac{1}{t} \geq 1$ and therefore (\ref{EQU Verteilungsfunktion log-gamma}) applies and we obtain
    $$
     \Pv\left(|\xi| \geq \frac{1}{t}\right) = t^p.
    $$
  Hence
    $$
      (I) = \int_0^1 \frac{1}{t}t^p dt = \left[\frac{1}{p}t^p\right]_0^1 = \frac{1}{p}.
    $$      
  In $(II)$, we have $\frac{1}{t} \leq 1$ and therefore
    $$
      \Pv\left(|\xi| \geq \frac{1}{t}\right) = \int_{\frac{1}{t}}^{\infty} f_{\xi}(x) dx = \int_{1}^{\infty} f_{\xi}(x) dx
      = \int_1^{\infty} px^{-p-1} dx = [-x^{-p}]_1^{\infty} = 1. 
    $$       
  So we obtain
    $$
      (II) = \int_1^s \frac{1}{t} dt = \ln(s).
    $$
  Therefore
    $$
      \int_0^s \frac{1}{t} \Pv\left(|\xi| \geq \frac{1}{t}\right) dt = \ln(s) + \frac{1}{p}.
    $$
  We now calculate
    $$
      \int_0^s \int_{1/t}^{\infty} \Pv(|\xi| \geq u) du ~ dt.
    $$
  Again, we have
    $$
      \int_0^s \int_{1/t}^{\infty} \Pv(|\xi| \geq u) du ~ dt = 
      \underbrace{\int_0^1 \int_{1/t}^{\infty} \Pv(|\xi| \geq u) du ~ dt}_{(III)} + 
      \underbrace{\int_1^s \int_{1/t}^{\infty} \Pv(|\xi| \geq u) du ~ dt}_{(IV)}.
    $$
  In part $(III)$, we have $\frac{1}{t} \geq 1$ and hence by (\ref{EQU Verteilungsfunktion log-gamma})
    $$
      \Pv(|\xi| \geq u) = u^{-p}.
    $$
  So we obtain 
    $$
      \int_{\frac{1}{t}}^{\infty} u^{-p} = \frac{1}{p-1}t^{p-1}. 
    $$
  Therefore
    $$
      (III) = \int_0^1 \int_{1/t}^{\infty} \Pv(|\xi| \geq u) du ~ dt = \int_0^1 \frac{1}{p-1}t^{p-1} dt = \frac{1}{p(p-1)}.
    $$
  In part $(IV)$, we have $\frac{1}{t} \leq 1$, so we get
    $$
      \int_{\frac{1}{t}}^{\infty} \Pv(|\xi| \geq u) du =
      \underbrace{\int_{\frac{1}{t}}^{1} \Pv(|\xi| \geq u) du}_{(IV.1)}
      + \underbrace{\int_{1}^{\infty} \Pv(|\xi| \geq u) du}_{(IV.2)}. 
    $$
  Since in $(IV.1)$ we have $u\leq 1$, we obtain
    $$
      (IV.1) = \int_{\frac{1}{t}}^1 \int_u^{\infty} f_{\xi}(x) dx ~ du = \int_{\frac{1}{t}}^1 \underbrace{\int_1^{\infty} f_{\xi}(x) dx}_{=1} ~ du 
      = \int_{\frac{1}{t}}^1 1 du = 1-\frac{1}{t}.
    $$
  In $(IV.2)$, we have $u\geq 1$ and therefore by (\ref{EQU Verteilungsfunktion log-gamma})
    $$
      \Pv(|\xi| \geq u) = u^{-p}.
    $$
  Hence
    $$
      (IV.2) = \int_1^{\infty} u^{-p} du = \frac{1}{p-1}.
    $$
  So
    $$
      (IV) = \int_1^s \int_{1/t}^{\infty} \Pv(|\xi| \geq u) du ~ dt = \int_1^s 1-\frac{1}{t} + \frac{1}{p-1} dt
      = \frac{p}{p-1}(s-1) - \ln(s).
    $$
  Altogether, we have
    $$
      M_{\xi}(s) = \ln(s)+\frac{1}{p} + \frac{1}{p(p-1)} + \frac{p}{p-1}(s-1) - \ln(s),
    $$
  i.e. for $s>1$ we have $M_{\xi}(s)=\frac{p}{p-1}s - 1$.                                          
\end{proof}

An Orlicz norm $\norm{\cdot}_M$ is uniquely determined on the interval $[0,s_0]$ where $M(s_0)=1$. Therefore, we obtain the following interesting corollary.

\begin{cor}
  Let $\xi_1,\ldots,\xi_n$ be i.i.d. copies of a $\log-\gamma_{1,2}$ distributed random variable $\xi$. Then, for all $x\in\R^n$,
    $$
      \Ex \max_{1\leq i \leq n} \abs{x_i\xi_i} \sim \norm{x}_2.
    $$  
\end{cor}

In fact, this is interesting since one would assume standard gaussians to generate the $\ell_2$-norm. In fact, the norm generated by Gaussians is far from being the $\ell_2$-norm.\\
${}$\\

Naturally now the question arises, can we prove Theorem \ref{thm:p/q-stabil} and Theorem \ref{thm:p-stabil/gauss} also in case that $p=2$, this means in the case that the random variables $\xi_i$, $i=1,...,n$, are independent $\log-\gamma_{1,2}$ distributed. We can do so, as provided in the following.

\begin{satz}\label{thm:p/2-stabil} Let $p\in(1,2)$, let $\xi_i$, $i=1,...,n$ be independent copies of a $\log-\gamma_{1,2}$ distributed random variable $\xi$ on $(\Omega_1,\mathfrak{A}_1,\Pv_1)$ and let $\eta_j$, $j=1,...,m$ be independent p-stable copies of a random variable $\eta$ on $(\Omega_2,\mathfrak{A}_2,\Pv_2)$. Then, for all $(a_{ij})_{i,j}\in\R^{n\times m}$, 
	$$\Ex_{\Omega_1}\Ex_{\Omega_2}\max\limits_{i,j}\left|a_{ij}\eta_j(\omega_2)\xi_i({\omega_1})\right|\sim_{p}\norm{\left(\norm{\left(a_{ij}\right)_{i=1}^n}_2\right)_{j=1}^m}_p.$$
\end{satz}

\begin{proof}
  We follow the proof of Theorem \ref{thm:p/q-stabil}. Therefore we have
    $$
      \Ex_{\Omega_1}\Ex_{\Omega_2} \max_{1\leq i \leq n} \max_{1\leq j \leq n } \abs{a_{ij} \xi_i(\omega_1) \eta_j(\omega_2)}
      \sim
      \Ex_{\Omega_1} \norm{\left(\max_{1\leq i\leq n}\abs{a_{ij}\xi_i(\omega_1)}\right)_{j=1}^n}_p.
    $$
  As before, we have to show that $M_{\xi^p}(s) \sim M_{\xi}(s^{1/p})=s^{2/p}$. We calculate $M_{\xi^p}(s)$ and start with $s\leq 1$. 
  Since for all $y\geq 1$
    $$
      \Pv(\abs{\xi} \geq y) = \int_y^{\infty} f_{\xi}(x) dx = y^{-2},  
    $$
  we obtain  
    \begin{eqnarray*}
      M_{\xi^p}(s) & = & \int_0^s \frac{1}{t} \Pv(\abs{\xi} \geq t^{-1/p}) + 
      \int_{\frac{1}{t}}^{\infty} \Pv(\abs{\xi}\geq u^{1/p}) du ~dt \\
      & = & \int_0^s t^{2/p-1} +  \left[-\frac{1}{2/p-1}u^{-2/p+1}\right]_{\frac{1}{t}}^{\infty} dt \\
      & = & \frac{2}{2-p}\int_0^s t^{2/p-1} dt \\
      & = & \frac{p}{2-p}s^{2/p}.
    \end{eqnarray*}
  So for all $s\leq 1$ we have $M_{\xi^p}(s) = \frac{p}{2-p}s^{2/p} = \frac{p}{2-p}M_{\xi}(s^{1/p})$. Since $\frac{p}{2-p}>1$, the case 
  $0\leq s \leq 1$ suffices because $M_{\xi^p}(1)>1$ and therefore the Orlicz norm $\norm{\cdot}_{M_{\xi^p}}$ is uniquely determined on 
  this interval.   
\end{proof}

\begin{satz}\label{thm:2-stabil/gauss} Let $\xi_i$, $i=1,...,n$ be independent copies of a $\log-\gamma_{1,2}$ distributed random variable $\xi$ on $(\Omega_1,\mathfrak{A}_1,\Pv_1)$ and let $\eta_j$, $j=1,...,m$ be independent copies of a standard gaussian random variable $\eta$ on $(\Omega_2,\mathfrak{A}_2,\Pv_2)$. Then, for all $(a_{ij})\in\R^{n\times m}$, 
	$$\Ex_{\Omega_1}\Ex_{\Omega_2}\max\limits_{i,j}\left|a_{ij}\eta_j(\omega_2)\xi_i({\omega_1})\right|\sim\norm{\left(\norm{\left(a_{ij}\right)_{i=1}^n}_{M_{\xi}}\right)_{j=1}^m}_2,$$
	where $\norm{\cdot}_{M_{\xi}}$ denotes again the Orlicz norm given by the Orlicz function
	\begin{equation}\label{eqn:m2}
	  M_{\xi}(s) = 
	   \left\{ 
	    \begin{array}{ll}
	                 0 & ,\hbox{if}\quad s=0  \\
	                 e^{-\frac{3}{2s^2}} & ,\hbox{if}\quad s\in (0,1)\\
	                 e^{-\frac{3}{2}}(3s-2)  & ,\hbox{if}\quad s\geq 1.
	    \end{array} 
	   \right.
	 \end{equation}
\end{satz}

The proof works exactly in the same way as the proof of Theorem \ref{thm:p-stabil/gauss}. \\

~\\
~\\
~\\
~\\
{\bf Joscha Prochno}\\
Mathematisches Seminar\\
Christian-Albrechts-Universit\"at zu Kiel\\
Ludewig Meyn Str. 4\\
24098 Kiel, Germany\\
{\em prochno@math.uni-kiel.de}\\
and\\
Department of Mathematical and Statistical Sciences\\
University of Alberta\\
605 Central Academic Building\\
Edmonton, Alberta\\
Canada T6G 2G1\\
{\em prochno@ualberta.ca}\\
~\\
~\\
{\bf Stiene Riemer}\\
Mathematisches Seminar\\
Christian-Albrechts-Universit\"at zu Kiel\\
Ludewig Meyn Str. 4\\
24098 Kiel, Germany\\
{\em riemer@math.uni-kiel.de}\\




\end{document}